\documentclass[twoside,a4paper]{amsart}

\usepackage{amssymb}
\usepackage{amsmath}

\newtheorem{theorem}{Theorem}[section]

\newtheorem{claim}[theorem]{Claim}

\newtheorem{corollary}[theorem]{Corollary}

\newtheorem{lemma}[theorem]{Lemma}

\newtheorem{proposition}[theorem]{Proposition}
\newtheorem{question}[theorem]{Question}

\theoremstyle{definition}

\theoremstyle{theorem}

\numberwithin{equation}{section}

\newcommand{\dom}{\operatorname{dom}}
\newcommand{\ran}{\operatorname{ran}}
\newcommand{\Ult}{\operatorname{Ult}}

\newcommand{\RCA}{\mathsf{RCA}}

\newcommand{\RAN}{\mathsf{RAN}}

\newcommand{\PA}{\mathrm{PA}}
\newcommand{\CARD}{\mathrm{CARD}}
\newcommand{\WPHP}{\mathrm{WPHP}}
\newcommand{\GPHP}{\mathrm{GPHP}}

\newcommand{\PHP}{\mathrm{PHP}}
\newcommand{\FRT}{\mathrm{FRT}}
\newcommand{\fairFRT}{\mathrm{fairFRT}}

\begin{document}

\title{Finite Combinatorics and Fragments of Arithmetic}

\begin{abstract}
In fragments of first order arithmetic, definable maps on finite domains could behave very differently from finite maps.
Here combinatorial properties of $\Sigma_{n+1}$-definable maps on finite domains are compared in the absence of $B\Sigma_{n+1}$.
It is shown that $\GPHP(\Sigma_{n+1})$ (the $\Sigma_{n+1}$-instance of Kaye's General Pigeonhole Principle) lies strictly between $\CARD(\Sigma_{n+1})$ and $\WPHP(\Sigma_{n+1})$ (Weak Pigeonhole Principle for $\Sigma_{n+1}$-maps), and also that $\FRT(\Sigma_{n+1})$ (Finite Ramsey's Theorem for $\Sigma_{n+1}$-maps) does not imply $\WPHP(\Sigma_{n+1})$.
\end{abstract}


\author{Wang Wei}

\address{Institute of Logic and Cognition and Department of Philosophy\\
Sun Yat-Sen University\\
Guangzhou, China}
\email{wwang.cn@gmail.com}
\email{wangw68@mail.sysu.edu.cn}

\thanks{This project is partially supported by Grant 2023A1515010892 from Guangdong Basic and Applied Basic Research Foundation, 
and by Grant 22JJD110002 from the Ministry of Education, China.
The author thanks Chitat Chong and Yue Yang for their offerring the inclusion of some results from an unfinished joint project, Prof. Beklemishev for his many helpful comments, and the referee for patience in reading sloppy proofs and composing a very detailed report.}

\maketitle

\section{Introduction}\label{s:Introduction}

There is an interesting relationship between finite combinatorics and fragments of first order arithmetic,
which is the topic of this article.
I shall briefly recall some history as motivation.

A usual statement of finite Pigeonhole Principle (PHP for short) is that there does not exist an injection from any
$z+1 = \{0,1,\ldots,z\}$ to $z = \{0,1,\ldots,z-1\}$.
Dimitracopoulos and Paris \cite{Dimitracopoulos.Paris:1986} introduced a variant of PHP (so-called $\PHP_1(\Sigma_n)$) over models of arithmetic, where the maps are just $\Sigma_n$-definable. So such maps may not exist in models which do not satisfy enough induction. They also introduced several other variants of PHP in a similar vein and studied these variants in fragments of arithmetic. For example, they proved that $\PHP_1(\Sigma_{n+1})$ and $B\Sigma_{n+1}$ are equivalent over $\PA^- + I\Sigma_0$.

On the other hand, Kaye \cite{Kaye:1995.NDJFL} tried to use combinatorics to axiomatize $\kappa$-like models of arithmetic.
Recall that a model $M$ with a linear order $<$ is $\kappa$-like for a cardinal $\kappa$,
 if $|M| = \kappa$ and $|\{a \in M: M \models a < b\}| < \kappa$ for every $b \in M$.
Among the formalizations of combinatorial principles in \cite{Kaye:1995.NDJFL}, there are two of particular interest.
The first is $\CARD_2$ which is a second order statement of the non-existence of injections from $M$ to any $b$,
 and the second $\GPHP_2$ is also second order and states that for every $a$ there exists $b$
 s.t. there is no injection from $b$ to $a$.
Kaye proved that a model $M \models \PA$ is $\kappa$-like for some $\kappa$ iff $M$ satisfies $\CARD_2$,
 and that $M \models \PA$ is $\kappa$-like for a limit cardinal $\kappa$ iff $M$ satisfies $\GPHP_2$.
Kaye also studied $\kappa$-like models of fragments of $\PA$ in \cite{Kaye:1997.JLMC},
 where it is proved that every $M \models B\Sigma_n + \neg I\Sigma_n$ ($n > 1$) is elementarily equivalent to a $\kappa$-like model if $\kappa$ is singular.
As a consequence, $B\Sigma_n + \neg I\Sigma_n$ ($n > 1$) implies $\GPHP$,
 which is the fragment of $\GPHP_2$ mentioning only first order definable maps.

In recent years, fragments of $\CARD_2$ and related formalizations of combinatorial principles are found useful in reverse mathematics.
The fragment $\CARD(\Sigma_2)$ of $\CARD_2$, which is the restriction of $\CARD_2$ to $\Sigma_2$-definable maps, is used in several occasions (e.g., \cite{Seetapun.Slaman:1995} and \cite{Conidis.Slaman:2013.JSL}) to show that certain instances or variants of Ramsey's Theorem are not arithmetically conservative over $\RCA_0$.
B\'{e}langer et al. \cite{BCWWY:2021} introduced $\WPHP(\Sigma_2)$, which lies strictly between $\CARD(\Sigma_2)$ and $\PHP_1(\Sigma_2)$, and used it to separate some interesting second order statements in reverse mathematics.

In this article, I shall study the relation between $\CARD(\Sigma_n)$, $\GPHP(\Sigma_n)$ and $\WPHP(\Sigma_n)$.
Besides, I shall also formalize finite Ramsey's theorem (FRT for short) in similar mode as Dimitracopoulos and Paris did to PHP and compare the resulted formalizations with $\GPHP(\Sigma_n)$ etc.
The organization of the remaining parts is as follows.
Section 2 recalls the detailed definitions of terms mentioned above and some known relations between them, and also introduces some formalizations ($\FRT^e_k(\Sigma_n)$ and a weaker $\fairFRT^e_k(\Sigma_n)$) of FRT. 
Section 3 contains some basic results about the formalized combinatorial principles.
In section 4, $\CARD(\Sigma_{n+1})$ and $\GPHP(\Sigma_{n+1})$ are separated.
Section 5 presents the separation of $\FRT^e_k(\Sigma_n)$ and $\WPHP(\Sigma_n)$.
Section 6 discusses $\FRT^e_k(\Sigma_{n+1})$ and $\WPHP(\Sigma_{n+1})$ in the context of $P\Sigma_n$.
Finally, section 7 concludes the article with some open questions.

\section{Formalizations}\label{s:Formalizations}

The terminology here largely follows Kaye's monograph \cite{Kaye:MOPA}.
Some concepts from combinatorics are needed.
A \emph{coloring} is just a function, and a $k$-coloring is a function with range contained in $k = \{0,\ldots,k-1\}$.
For a given set $X$ and a positive integer $e$, $[X]^e$ is the collection of subsets of $X$ with exactly $e$ many elements.
Usually $[X]^1$ is identified with $X$, and $[X]^e$ is occasionally identified with $\{(x_0,\ldots,x_{e-1}): x_i \in X, x_0 < \ldots < x_{e-1}\}$.
Given a coloring $C$ defined on some $[X]^e$, 
 a set $H$ is \emph{homogeneous} or \emph{$C$-homogeneous} if $H \subseteq X$ and $C$ is constant on $[H]^e$.

We use Erd\"{o}s notation
\[
  y \to (x)^e_k
\]
as a shorthand of the statement that every $k$-coloring of $[y]^e$ admits a homogeneous set of size $x$.
In the context of a model $M \models I\Sigma_1$, the colorings and homogeneous sets in $y \to (x)^e_k$ are supposed to be $M$-finite.
It is widely known that 
\[
  I\Sigma_1 \vdash \forall x,e,k \exists y (y \to (x)^e_k),
\]
as in H\'ajek and Pudl\'ak \cite[Theorem II.1.10]{HP:1998.book} or Kaye \cite[\S 1.3]{Kaye:1997.JLMC}.
Hence the following $\Sigma_1$-definable function is $I\Sigma_1$-provably total and thus primitive recursive,
\[
  R(x,e,k) = \min\{y: y \to (x)^e_k\}.
\]

Recall that $\CARD_2$ is defined by Kaye \cite{Kaye:1995.NDJFL} as
\begin{multline*}
  \forall F \forall a \big(\forall x \exists y < a \ (x,y) \in F \to \\
  \exists y < a \exists x_0, x_1 (x_0 \neq x_1 \wedge (x_0,y) \in F \wedge (x_1,y) \in F)\big).  
\end{multline*}
Kaye also defined $\CARD$ to be $\CARD_2$ restricted to first order definable $F$'s, and $\CARD(\Sigma_n)$ to be $\CARD_2$ restricted to $\Sigma_n$ definable $F$'s.

In \cite{Kaye:1995.NDJFL}, Kaye also introduced what he called \emph{a generalized pigeonhole principle} or $\GPHP_2$ as follows,
\[
  \forall x \exists y \forall F (F \text{ is not an injection from } y \text{ to } x).
\]
Similarly the restriction of $\GPHP_2$ to first order (or just $\Sigma_n$) definable functions is denoted by $\GPHP$ (or $\GPHP(\Sigma_n)$).
Belanger et al. \cite{BCWWY:2021} used a modified Erd\"{o}s' notation
\[
  \Sigma_n: y \to (z)^e_x
\] 
as an abbreviation of the first-order sentence that every $\Sigma_n$-definable map from $[y]^e$ to $x$ has a homogeneous set of size $z$.
So, $\GPHP(\Sigma_n)$ can be written as
\[
  \forall x \exists y (\Sigma_n: y \to (2)^1_x).
\]

Belanger et al. \cite{BCWWY:2021} defined \emph{Weak Pigeonhole Principle for $\Sigma_n$-definable maps} or $\WPHP(\Sigma_n)$ to be
\[
  \forall x > 0 (\Sigma_n: 2x \to (2)^1_x),
\]
i.e., there is no $\Sigma_n$-injection from any $2x > 0$ to $x$.

Some obvious relations and known results about the above formalized combinatorial principles are summarized below.

\begin{theorem}\label{thm:knowns}
Let $n$ be a positive integer and $\kappa$ be a singular cardinal.
\begin{enumerate}
  \item $B\Sigma_{n+1} \to \WPHP(\Sigma_{n+1}) \to \GPHP(\Sigma_{n+1}) \to \CARD(\Sigma_{n+1})$.
  \item $I\Sigma_n \not\vdash \CARD(\Sigma_{n+1})$.
  \item (Kaye \cite{Kaye:1997.JLMC,Kaye:1995.NDJFL}) Every $M \models B\Sigma_{n+1} + \neg I\Sigma_{n+1}$ is elementarily equivalent to a $\kappa$-like $N$.
  Thus $B\Sigma_{n+1} + \neg I\Sigma_{n+1} \vdash \GPHP$.
  \item (Belanger et al. \cite{BCWWY:2021}) $\WPHP(\Sigma_{n+1})$ lies strictly between $B\Sigma_{n+1}$ and $\CARD(\Sigma_{n+1})$.
\end{enumerate}
\end{theorem}

\begin{proof}
(1) 
By Dimitracopoulos and Paris \cite{Dimitracopoulos.Paris:1986}, $B\Sigma_{n+1}$ is equivalent to $\forall x (\Sigma_{n+1}: x+1 \to (2)^1_x)$ over $I\Sigma_1$. So $B\Sigma_{n+1}$ implies $\WPHP(\Sigma_{n+1})$ by the definition of $\WPHP(\Sigma_{n+1})$.

It is obvious that
\[
  \forall x > 0 (\Sigma_{n+1}: 2x \to (2)^1_x) \to \forall x \exists y (\Sigma_{n+1}: y \to (2)^1_x),
\]
i.e., $\WPHP(\Sigma_{n+1}) \to \GPHP(\Sigma_{n+1})$.

Working in a model $M \models \GPHP(\Sigma_{n+1})$.
Fix $a \in M$.
As $\Sigma_{n+1}: b \to (2)^1_a$ holds for some $b$, there cannot exist a $\Sigma_{n+1}$-definable injection $F: M \to a$. Thus $\CARD(\Sigma_{n+1})$ holds.

(2) can be found in several sources, e.g., \cite[Lemma 3.4]{Seetapun.Slaman:1995} and \cite[Proposition 3.2]{Kaye:1995.NDJFL}.

For (3), given $M$, the model $N$ is constructed by Kaye \cite[Theorem 2.4]{Kaye:1997.JLMC}.
By Kaye \cite[Proposition 3.4]{Kaye:1995.NDJFL}, $N \models \GPHP$.
As $M$ and $N$ are elementarily equivalent, $M \models \GPHP$ as well.

(4) is by \cite[Theorems 5.2 and 5.10]{BCWWY:2021}.
\end{proof}

On the other hand, the following question posted by Kaye \cite{Kaye:1995.NDJFL} seems still open.

\begin{question}
Are $\CARD$ and $\GPHP$ equivalent (over $I\Sigma_1$) ?
\end{question}

Though this question is not answered here, section 3 contains a separation of $\CARD(\Sigma_{n+1})$ and $\GPHP(\Sigma_{n+1})$ over $I\Sigma_n$ ($n > 0$).
And the main theorem of section 5 implies the separation of $\WPHP(\Sigma_{n+1})$ and $\GPHP(\Sigma_{n+1})$.

Here we are also interested in the variant of Finite Ramsey Theorem for $k$-colorings of $e$-tuples which are definable.
Let $\FRT^e_k(\Sigma_n)$ be the first-order formula stating that
\[
  \forall x \exists y (\Sigma_n: y \to (x)^e_k).
\]
Let $\FRT^e(\Sigma_n)$ denote $\forall k (\FRT^e_k(\Sigma_n))$,
and let $\FRT(\Sigma_n)$ denote $\forall e \FRT^e(\Sigma_n)$.

One may associate a two-player game to a finite combinatorial principle like Finite Ramsey Theorem.
For example, in the game associated to $\FRT^e_k(\Sigma_n)$,
the first player raises a challenge by providing a $\Sigma_n$-definable $k$-coloring of $[y]^e$,
and the second player responds with a finite set $H$ of size $x$.
If $[H]^e$ is homogeneous, the second player wins the game, otherwise the first player wins.
However, people may argue that such a game is unfair as the winning condition is more limited for the second player than for the first.
So, the following variants of Finite Ramsey Theorem may look fairer and are thus called $\fairFRT^e_k(\Sigma_n)$.
\begin{multline*}
  \forall x \exists y \forall F \in \Sigma_n (F \text{ is a } k \text{-coloring of } [y]^e \to \\
  \exists H \in \Sigma_n(H \text{ is an injection from } x \text{ to } y \wedge
   [\ran H]^e \text{ is homogeneous for } F)).  
\end{multline*}
Let $\fairFRT^e(\Sigma_n)$ denote $\forall k (\fairFRT^e_k(\Sigma_n))$.

We conclude this section with a probably well-known proposition of cofinal extensions.
Indeed, Paris proved a special case in (iv) of the proof of \cite[Theorem 11]{Paris:1981}.
This proposition is implicitly or explicitly used in several places.

\begin{proposition}\label{pro:cf-ext-induction}
Suppose that $n > 0$, $M \models I\Sigma_n$ and $N$ is a $\Sigma_{n+1}$-elementary cofinal extension of $M$.
Then $N \models I\Sigma_n$.
\end{proposition}

\begin{proof}
Let $\varphi(x,y)$ be a $\Sigma_n$-formula with all free variables displayed and let $a \in N$ be s.t. $N \models \exists x \varphi(x,a)$.
Fix $b \in M$ greater than $a$.
By $M \models I\Sigma_n$, there exist $c \in M$ and an $M$-finite $B \subseteq [0,b]^M$ s.t.
\[
  M \models \forall y \leq b (y \in B \leftrightarrow \exists x \varphi \leftrightarrow \exists x < c\ \varphi).
\]
By elementarity, the above also holds in $N$ and $N \models a \in B$.
By $M \models I\Sigma_n$ again, there is another $M$-finite set $F$ s.t. in $M$
\[
  F = \{(x,y): x < c \wedge y \in B \wedge x \text{ is the least number satisfying } \varphi\}.
\]
and $M \models \forall y \in B \exists x < c ((x,y) \in F)$.
By elementarity, 
\[
  N \models \forall y \in B \exists x < c ((x,y) \in F),
\]
and thus there exists $x_0$ s.t. $N \models (x_0,a) \in F$.
By elementarity again, $x_0$ is the least number satisfying $\varphi(x,a)$ in $N$.
\end{proof}

\section{Basics}\label{s:Basics}

This section presents some basic results.
Firstly, it is easy to observe the followings.

\begin{proposition}\label{pro:easy-implications}
Over $\PA^-$,
\begin{enumerate}
  \item $I\Sigma_1 \vdash \FRT(\Sigma_1)$.
  \item $I\Sigma_1$ proves the following implications
  \[
    \WPHP(\Sigma_n) \to \GPHP(\Sigma_n) \to \CARD(\Sigma_n).
  \]
  \item $I\Sigma_1$ proves that if $0 < e \leq e' \leq e'' \in \omega$ and $0 < k \leq k' \leq k'' \in \omega$ then 
  \[
    \FRT^{e''}_{k''}(\Sigma_n) \to \FRT^{e'}_{k'}(\Sigma_n) \to
    \fairFRT^{e'}_{k'}(\Sigma_n) \to \fairFRT^{e}_{k}(\Sigma_n).
  \]
  \item $I\Sigma_1 \vdash \fairFRT^{1}(\Sigma_n) \to \GPHP(\Sigma_n)$.
  \item $B\Sigma_{n+1} \vdash \FRT(\Sigma_{n+1})$.
\end{enumerate}
\end{proposition}

\begin{proof}
Let $M \models I\Sigma_1$ and $0 < e \in M$.
Working in $M$, let $y = R(x,e,k)$ and let $F: [y]^e \to k$ be $\Sigma_1$-definable.
By $I\Sigma_1$, $F$ is $M$-finite.
So there exists an $F$-homogeneous set of size $x$ by $y \to (x)^e_k$.
Hence $M \models \FRT(\Sigma_1)$.

If $M \models B\Sigma_{n+1}$ and $F: [y]^e \to k$ is $\Sigma_{n+1}$-definable where $y = R(x,e,k)$ as above,
then $F$ is also $M$-finite by $B\Sigma_{n+1}$.
So $F$ admits a homogeneous set of size $x$ by $y \to (x)^e_k$ again.
Thus $M \models \FRT(\Sigma_{n+1})$.

The above proves (1) and (5).
The remaining is obvious.
\end{proof}

To prove the next two theorems, the lemma below is needed.
Intuitively, it turns a $\Sigma_{n+1}$-injection into a $\Sigma_{n+1}$-enumeration of its domain which behaves nicely.

\begin{lemma}\label{lem:ISigma-tame-inj}
Let $M \models I\Sigma_n$ with $n > 0$.
Suppose that $F$ is a $\Sigma_{n+1}$-injection with range bounded by $b$.
Then in $M$, there exists a $\Sigma_{n}$-function $G: b \times M \to M$ s.t.
\begin{enumerate}
  \item $G(\cdot,s)$ is injective on $b$ for all $s$;
  \item $\dom F = \{x: \exists i < b (x = \lim_s G(i,s))\}$;
  \item If $i < b$ and $\lim_s G(i,s)$ exists then there exists $s$ with $G(j,s) = G(j,t)$ for all $j \leq i$ and $t > s$.
\end{enumerate}
Note that $\{i < b: \lim_s G(i,s) \text{ exists}\}$ is downward closed by (3).
\end{lemma}

\begin{proof}
We work in $M$.
By Shoenfield's Limit Lemma (as formulated in \cite[Lemma 2.1]{BCWWY:2021}), there exists a $\Sigma_n$-function $F_0$ s.t. $x \in \dom F$ iff $\lim_s F_0(x,s)$ exists and equals $F(x)$.
We may assume that 
\begin{equation}\label{eq:ISigma-tame-inj}
  x \geq s \to F_0(x,s) = s.
\end{equation}

Define
\[
  m(x,s) = \min\{r \leq s: \forall t \leq s (r < t \to F_0(x,r) = F_0(x,t))\}.
\]
By $I\Sigma_n$, $m$ is a well-defined $\Sigma_n$-function.
Moreover, $x \in \dom F$ iff $\bar{m}(x) = \lim_s m(x,s)$ exists
and $F_0(x,t) = F_0(x,\bar{m}(x))$ for all $t > \bar{m}(x)$.
By \eqref{eq:ISigma-tame-inj} and the definition of $m$,
\begin{equation}\label{eq:ISigma-tame-inj-m}
  x \geq s \to m(x,s) = s \quad \text{ and } \quad
  t > s \wedge m(x,t) \neq m(x,s) \to m(x,t) > s.
\end{equation}
By $I\Sigma_n$ and $m(x,s) \leq s$, the following fragment of $m$ is $M$-finite
\[
  \tilde{m}(s) = \{(x,m(x,s)): x < s\} = \{(x,m(x,s)): x < s, m(x,s) \leq s\}.
\]
Moreover, the map $s \mapsto \tilde{m}(s)$ is $\Sigma_n$-definable.

Using $\tilde{m}$, define a $\Delta_n$ total order $\prec_s$ on $s$ as follows
\[
  x \prec_s y \ \Leftrightarrow \ m(x,s) < m(y,s) \vee (m(x,s) = m(y,s) \wedge x < y).
\]
For $x < s$, let
\[
  P(x,s) = \{y < s: y \prec_s x\}.
\]
By $I\Sigma_n$, each $P(x,s)$ is $M$-finite and $P$ is also a well-defined $\Sigma_n$-function.
Define $G$ on $b \times M$ as follows,
\[
  G(i,s) =
  \begin{cases}
    x & x < s, m(x,s) < s, |P(x,s)| = i \\
    s + i & \text{the above fails for all } x < s.
  \end{cases}
\]
To see $G$ is well-defined, fix $x < y$ s.t. $s > \max\{x,y,m(x,s),m(y,s)\}$.
Then either $x \prec_s y$ or $y \prec_s x$.
Thus $|P(x,s)| \neq |P(y,s)|$.
This shows that for each $i$ there exists at most one $x$ satisfying the first case of the definition of $G(i,s)$.
With the help of $\tilde{m}$ above, it is clear that $G$ is $\Sigma_n$-definable.

Below, we verify that $G$ satisfies (1-3).

(1)
Suppose that $i < j < b$.
If $G(i,s) \geq s$ then $G(i,s) \neq G(j,s)$.
If $G(i,s) = x < s$ then $|P(x,s)| = i$ and thus $G(j,s) \neq x$.
This shows that $G(\cdot,s)$ is injective on $b$.

(2)
Suppose that $x \in \dom F$.
Let $s$ be s.t. $s > x$ and $F_0(x,t) = F_0(x,s)$ for all $t > s$.
Then $\bar{m}(x) = m(x,s)$. We may assume that $s > \bar{m}(x) = m(x,s)$.
Let
\[
  Q(x,s) = \{y \in P(x,s): \exists t > s (m(y,t) > m(y,s))\},
\]
and let $t > s$ be s.t. $m(y,t) > m(y,s)$ for all $y \in Q(x,s)$.
Note that $Q(x,s)$ is a bounded $\Sigma_n$-set and thus $M$-finite.
So $t$ exists by $B\Sigma_n$.
For $y \in P(x,t)$, $m(y,t) < m(x,t) = m(x,s) < s$, 
 thus $m(y,t) = m(y,s)$ by \eqref{eq:ISigma-tame-inj-m} and $y \in P(x,s) - Q(x,s)$;
 for $y \in P(x,s) - Q(x,s)$, $y < s < t$ and $m(y,t) = m(y,s) < m(x,s) = m(x,t)$,
 thus $y \in P(x,t)$.
Hence $P(x,t) = P(x,s) - Q(x,s)$.
Moreover, $\bar{m}(y) = m(y,s)$ and $F(y) = F_0(y,s)$ if $y = x$ or $y \in P(x,t)$.
As $F$ is injective and has range bounded by $b$, $|P(x,t)| = i < b$ and $G(i,t') = x$ for all $t' > t$.
So $x = \lim_s G(i,s)$.

On the other hand, suppose that $x = \lim_s G(i,s)$ for some $i < b$.
Let $s > x$ be s.t. $G(i,t) = x$ for all $t \geq s$.
By the definition of $G$, $m(x,s) < s$.
If there were $t > s$ s.t. $F_0(x,t) \neq F_0(x,s)$,
then $m(x,t) = t$ for the least such $t$ and thus $G(i,t) \neq x$.
This shows that $F_0(x,t) = F_0(x,s)$ for all $t > s$ and $x \in \dom F$.

(3) 
Let $i < b$ and $s > x$ be s.t. $G(i,t) = x$ for all $t \geq s$.
By the proof of (2) above, $s > m(x,s) = m(x,t)$ 
and $P(x,s) \supseteq P(x,t)$ for all $t \geq s$.
If there were $j < i$ and $t > s$ s.t. $G(j,t) \neq G(j,s)$,
then $m(y,t) > s$ by \eqref{eq:ISigma-tame-inj-m} and $y \in P(x,s) - P(x,t)$ for $y = G(j,s)$.
It would follow that $|P(x,t)| < |P(x,s)| = i$ and $G(i,t) \neq x$.
So, $G(j,t) = G(j,s)$ for all $j \leq i$ and $t \geq s$.
\end{proof}

The first theorem here shows that $\GPHP$, $\FRT^1$ and $\fairFRT^1$ are locally equivalent.

\begin{theorem}\label{thm:GPHP-FRT}
Over $I\Sigma_n$ where $n > 0$,
\[
  \GPHP(\Sigma_{n+1}) \ \leftrightarrow\ \FRT^1(\Sigma_{n+1})
  \ \leftrightarrow\ \fairFRT^1(\Sigma_{n+1}) \ \leftrightarrow\ \FRT^1_2(\Sigma_{n+1}). 
\]
\end{theorem}

\begin{proof}
We work in a fixed $M \models I\Sigma_n$.

(1)
Suppose that $M \models \GPHP(\Sigma_{n+1})$.
Fix $x$ and $k$. 
By $\GPHP(\Sigma_{n+1})$, let $y$ be s.t.
\[
  \Sigma_{n+1}: y \to (2)^1_{kx}.
\]
Let $f: y \to k$ be $\Sigma_{n+1}$ definable over $M$.
By Shoenfield's Limit Lemma, there exists a $\Sigma_n$-definable $g: y \times M \to k$ s.t.
\[
  \forall z < y (f(z) = \lim_s g(z,s)).
\]
By \cite[Lemma 2.4]{BCWWY:2021}, there exist an $M$-finite $A \subseteq y$ and $s$ s.t.
$|A| \geq kx$ and 
\[
  \forall t > s, z \in A(g(z,s) = g(z,t) = f(z)).
\]
Let $i < k$ be s.t.
\[
  |\{z \in A: g(z,s) = i\}| \geq x.
\]
Then $H = \{z \in A: g(z,s) = i\}$ is $f$-homogeneous and $M$-finite of size $\geq x$. This shows that $M \models \FRT^1(\Sigma_{n+1})$.

(2)
Clearly, $\FRT^1(\Sigma_{n+1})$ implies both $\fairFRT^1(\Sigma_{n+1})$ and $\FRT^1_2(\Sigma_{n+1})$.

(3)
It follows from Proposition \ref{pro:easy-implications}(4)
that $I\Sigma_1 \vdash \fairFRT^1(\Sigma_{n+1}) \to \GPHP(\Sigma_{n+1})$.

(4)
It remains to prove that $\FRT^1_2(\Sigma_{n+1})$ implies $\GPHP(\Sigma_{n+1})$.
Assume that $M \models \neg \GPHP(\Sigma_{n+1})$.
So there exists $x$ 
 s.t. for every $y$ there exists a $\Sigma_{n+1}$-injection from $y$ into $x$.
Fix $y$.
Let $F: [y]^{2x} \cup y \to x$ be a $\Sigma_{n+1}$-injection.
Apply Lemma \ref{lem:ISigma-tame-inj} to $F$, we get a $\Sigma_n$-function $G$.

Fix $s$.
Let $A_{s,0} = B_{s,0} = \emptyset$.
For each $i < x$, if $G(i,s) \in y$ then simply let $A_{s,i+1} = A_{s,i}$ and $B_{s,i+1} = B_{s,i}$;
 if $G(i,s) \in [y]^{2x}$, then regard $G(i,s)$ as an $M$-finite subset of $y$ of size $2x$,
 and let $(a_{i,s},b_{i,s})$ be the least two elements in
\[
  G(i,s) - A_{s,i} \cup B_{s,i} \cup \{z < y: \exists j < i (z = G(j,s))\},
\]
and let $A_{s,i+1} = A_{s,i} \cup \{a_{i,s}\}$ and $B_{s,i+1} = B_{s,i} \cup \{b_{i,s}\}$.
Note that by induction
\[
  |A_{s,i}| = |B_{s,i}| = |\{j < i: G(j,s) \in [y]^{2x}\}|,
\]
and thus
\[
  |G(i,s)| = 2x > 2i \geq |A_{s,i} \cup B_{s,i} \cup \{z < y: \exists j < i (z = G(j,s))\}|.
\]
So the pair $(a_{i,s},b_{i,s})$ exists.
For $z < y$, define
\[
  C(z,s) = 
  \begin{cases}
    0 & z \in A_{s,x} \\
    1 & \text{otherwise.}
  \end{cases}
\]
Clearly $C: y \times M \to 2$ is $\Sigma_n$.

For each $z \in y$, there exist $s$ and $i < x$ s.t. $G(i,t) = z$ for all $t \geq s$.
By Lemma \ref{lem:ISigma-tame-inj}(3), we may also assume that $G(j,t) = G(j,s)$ for all $t \geq s$ and $j < i$.
By the construction of $C$, $C(z,t) = C(z,s)$ for all $t > s$.
Thus $\bar{C}(z) = \lim_s C(z,s)$ defines a $\Sigma_{n+1}$-coloring of $y$. 

Similarly, for each $z \in [y]^{2x}$ there exist $s$ and $i < x$ s.t. $G(i,t) = z$  and $G(j,t) = G(j,s)$ for all $t \geq s$ and $j < i$.
By the construction of $(a_{i,s},b_{i,s})$, $a_i = \lim_t a_{i,t} = a_{i,s}$ and $b_i = \lim_t a_{i,t} = b_{i,s}$ exist and form a pair in $z$.
Thus $\bar{C}(a_i) = 0 \neq 1 = \bar{C}(b_i)$ and $z$ is not $\bar{C}$-homogeneous.
Hence, $\bar{C}$ has no homogeneous $M$-finite set with cardinality $2x$.

This finishes showing that $M \not\models \FRT^1_2(\Sigma_{n+1})$.
\end{proof}

It turns out that $\fairFRT^1_2(\Sigma_{n+1})$ is strictly weaker than the combinatorial principles considered in the above theorem.

\begin{theorem}\label{thm:ISigma-fairFRT}
For $n > 0$, $I\Sigma_n \vdash \fairFRT^1_2(\Sigma_{n+1})$.
\end{theorem}

\begin{proof}
By Theorem \ref{thm:GPHP-FRT},
$
  I\Sigma_n + \GPHP(\Sigma_{n+1}) \vdash \fairFRT^1_2(\Sigma_{n+1}).
$

Below we work in a fixed
$
  M \models I\Sigma_n + \neg \GPHP(\Sigma_{n+1}).
$
Let
\[
  I = \{a \in M: \text{ there exists } b \text{ s.t. } \Sigma_{n+1}: b \to (2)^1_a\}.
\]
By $\neg \GPHP(\Sigma_{n+1})$, $I$ is a proper cut.
If $a \in I$ and $b \in M - I$ then $\Sigma_{n+1}: b \to (2)^1_a$.

We claim that $I$ is an additive cut.
Let $a \in I$, pick $b > a$ s.t. $\Sigma_{n+1}: b \to (2)^1_a$.
Let $F: 2b \to 2a$ be a $\Sigma_{n+1}$-map.
By \cite[Lemma 2.4(5)]{BCWWY:2021}, there exist $M$-finite $A_0 \subseteq [0,b-1]$ and $A_1 \subseteq [b,2b-1]$ s.t.
$|A_i| \geq a$ and $\{(x,F(x)): x \in A_i\}$ is also $M$-finite for $i < 2$.
So $F$ cannot be injective.

Fix $a$.
Let $b \in M - I$.
Fix a $\Sigma_{n+1}$-coloring $F: b \to 2$.
Apply Lemma \ref{lem:ISigma-tame-inj} to the inclusion maps $F^{-1}(i) \subseteq b$, we get a $\Sigma_{n+1}$-injection $\bar{G}_i$ from a downward closed subset of $b$ to $F^{-1}(i)$ for each $i < 2$.
There are three cases to consider.

(i)
If there exist $c_i \in I - \dom \bar{G}_i$ for $i < 2$, then we can define a $\Sigma_{n+1}$-injection from $b$ to $c_0 + c_1$ as follows,
\[
  H(x) =
  \begin{cases}
    \bar{G}_0^{-1}(x) & F(x) = 0 \\
    c_0 + \bar{G}_1^{-1}(x) & F(x) = 1.
  \end{cases}
\]
However, this is impossible, since $b \not\in I$ and $I$ is additive.

(ii)
There exist $i < 2$ and $c$ s.t. $c \in \dom \bar{G}_i - I$.
As $c \not\in I$, there exists a $\Sigma_{n+1}$-injection $H: a \to c$.
Then $\bar{G}_i \circ H$ is a $\Sigma_{n+1}$-injection from $a$ to $F^{-1}(i)$.

(iii)
Both $\dom \bar{G}_i = I$ ($i < 2$).
Define $H: b \to M$ by
\[
  H(x) =
  \begin{cases}
    2 \bar{G}_0^{-1}(x) & F(x) = 0 \\
    2 \bar{G}_1^{-1}(x) + 1 & F(x) = 1.
  \end{cases}
\]
As $I$ is additive and $\bar{G}_i$'s are injective, $H$ is a $\Sigma_{n+1}$-bijection from $b$ onto $I$.
Then $\bar{G}_i \circ H: b \to F^{-1}(i)$ is a $\Sigma_{n+1}$-injection for both $i < 2$.
As $b \not\in I$, there exists a $\Sigma_{n+1}$-injection from $a$ to $b$, and the composition of such an injection and $\bar{G}_i \circ H$ witnesses the existence of a $\Sigma_{n+1}$-injection from $a$ to $F^{-1}(i)$.

So we show that $M \models \fairFRT^1_2(\Sigma_{n+1})$.
\end{proof}

\begin{corollary}
Over $I\Sigma_n$ ($n > 0$), $\fairFRT^1_2(\Sigma_{n+1})$ is strictly weaker than either of $\fairFRT^1(\Sigma_{n+1})$, $\FRT^1_2(\Sigma_{n+1})$ and $\GPHP(\Sigma_{n+1})$.
\end{corollary}

\begin{proof}
By Theorem \ref{thm:GPHP-FRT}, both $\fairFRT^1(\Sigma_{n+1})$ and $\FRT^1_2(\Sigma_{n+1})$ are equivalent to $\GPHP(\Sigma_{n+1})$ over $I\Sigma_n$.
By Theorem \ref{thm:knowns}, all of the aboves are unprovable in $I\Sigma_n$.
So the corollary follows from Theorem \ref{thm:ISigma-fairFRT}.
\end{proof}

\section{Generalized Pigeonhole Principle and Cardinality Scheme}

This section is devoted to separate $\CARD(\Sigma_{n+1})$ and $\GPHP(\Sigma_{n+1})$.

\begin{theorem}\label{thm:CARD-GPHP}
$I\Sigma_n + \CARD(\Sigma_{n+1}) \not\vdash \GPHP(\Sigma_{n+1})$.
\end{theorem}

The proof of the above separation theorem uses an ultrapower construction invented by Paris \cite{Paris:1981}.

Let $M \models I\Sigma_1$ and $0 < a \in M$.
An \emph{$M$-ultrafilter on $a$} is an ultrafilter $U$ of the Boolean algebra $\mathcal{P}(a)^M$ consisting of $M$-finite subsets of the $M$-interval $[0,a-1]^M$.
If $f$ and $g$ are $M$-finite functions with domain $[0,a-1]^M$, then define
\[
  f \sim_U g \Leftrightarrow \{x \in M: M \models x < a \wedge f(x) = g(x)\} \in U,
\]
and
\[
  [f]_{U} = \{g \in M: f \sim_U g\}.
\]
Let $\Ult_U(M)$ be the natural structure defined on the set of all $[f]_U$'s,
and call it \emph{the ultrapower of $M$ by $U$}.

\begin{theorem}[Paris]\label{thm:Paris-Ult}
Let $U$ be an $M$-ultrafilter on a positive $a \in M$.
If $M \models I\Sigma_n$ 
then $\Ult_U(M) \models I\Sigma_n$ and $\Ult_U(M)$ is a $\Sigma_{n+1}$-elementary cofinal extension of $M$.
\end{theorem}

It is well-known that an ultrapower construction as above can also be viewed as a Skolem-hull construction inside a monster model.
This is stated and proved explicitly for the convenience of the reader.
The Skolem-hull viewpoint is useful in the proof of separation.

\begin{lemma}\label{lem:Ult-reprstn}
Let $M \models I\Sigma_1$ and $U$ be an $M$-ultrafilter on a positive $a \in M$.
There exist $N$ and $c \in N$, s.t., $M \subseteq_{cf} N$,
 \[
  N = \{f(c): f \in M \models f \text{ is a function with domain } a\}.
 \]
 and the mapping $[f]_U \mapsto f(c)$ is an isomorphism from $\Ult_U(M)$ onto $N$.
\end{lemma}

\begin{proof}
Let $\mathbb{M}$ be an $|M|^+$-saturated elementary extension of $M$.
Then there exists $c$ s.t. $\mathbb{M} \models c \in X$ for every $X \in U$.
Let
\[
  N = \{f(c) \in \mathbb{M}: f \in M \models f \text{ is a function with domain } a\}.
\]
Then $N$ induces a submodel of $\mathbb{M}$ and $M \subseteq_{cf} N$.
It is easy to verify that $[f]_U \mapsto f(c)$ defines an isomorphism from $\Ult_U(M)$ onto $N$.
\end{proof}

Denote the model $N$ in Lemma \ref{lem:Ult-reprstn} by $M[c]$.

\begin{lemma}\label{lem:CARD-lifting}
Let $M$ be a model of $I\Sigma_n$ for some $n > 0$.
Suppose that
\begin{itemize}
 \item There exists a $\Sigma_{n+1}^M$-injection from $M$ into some $a \in M$;
 \item $N = M[c]$ for some $c < a \in M$;
 \item and $N$ is a $\Sigma_{n+1}$-elementary cofinal extension of $M$.
\end{itemize}
Then there also exists a $\Sigma_{n+1}^N$-injection from $N$ into $a$.
\end{lemma}

\begin{proof}
Without loss of generality, assume that 
\begin{itemize}
 \item $F: M^2 \to M$ is defined by a $\Sigma_n^M$-formula $\varphi(x,s,y)$ in $M$,
 \item for each $s \in M$, $F(\cdot,s)$ is injective on $\{x \in M: F(x,s) < a\}$, and
 \item $\bar{F}(x) = \lim_s F(x,s)$ defines a $\Sigma_{n+1}^M$-injection from $M$ into $a$.
\end{itemize}
By elementarity, $\varphi$ also defines a function in $N$, denoted by $F^N: N^2 \to N$.
Moreover, 
\begin{equation}\label{eq:CARD-lifting}
  x \in M \to \lim_s F^N(x,s) = \lim_s F(x,s) = \bar{F}(x),
\end{equation}
and
\begin{equation}\label{eq:CARD-lifting-1}
  F^N(\cdot,s) \text{ is injective on } \{x \in N: F^N(x,s) < a\}.
\end{equation}

We work in $N$.
For each $x$, let $\psi(x,f,s,i)$ be the formula stating that
\[
  f(c) = x \wedge i < a \wedge \forall t > s \forall j \big( F^N(f,t) = j \to j = i \big).
\]
So $\psi$ is $\Pi_n^N$.
By $N = M[c]$ and \eqref{eq:CARD-lifting},
\[
  T(x) = \{(f,s,i): \psi(x,f,s,i)\} \neq \emptyset.
\]
Let $G(x)$ be the unique $i$ s.t. $(f,s,i) = \min T(x)$ for some $f,s$.
By Proposition \ref{pro:cf-ext-induction}, $N \models I\Sigma_n$, so $G(x)$ is defined and $G(x) < a$.
Moreover, $G$ is $\Sigma_{n+1}^N$.

Suppose that $x$ and $y$ are distinct, $G(x) = i$ and $G(y) = j$.
Let $f,s,g,t$ be s.t.
\[
  (f,s,i) = \min T(x),\quad (g,t,j) = \min T(y).
\]
Then $f \neq g$, as $f(c) = x \neq y = g(c)$.
By \eqref{eq:CARD-lifting-1} and the definition of $G$, for sufficiently large $u$,
\[
  i = F^N(f,u) \neq F^N(g,u) = j.
\]
This shows that $G$ is injective.

Hence, $G: N \to a$ is a $\Sigma_{n+1}^N$-injection.
\end{proof}

\begin{proof}[Proof of Theorem \ref{thm:CARD-GPHP}]
We begin with a countable model $M$ of $I\Sigma_n + \neg \CARD(\Sigma_{n+1})$.
Fix $a \in M$ s.t. there exists a $\Sigma_{n+1}^M$-injection from $M$ into $a$.
Also fix an increasing sequence $(b_k: k \in \mathbb{N})$ of elements of $M$ 
 and assume that $a < b_0$ and $(b_k: k \in \mathbb{N})$ is cofinal in $M$.
Note that all $b_k$'s are nonstandard as $a \not\in \mathbb{N}$.
Using \cite[Lemma 5.3]{BCWWY:2021} and Lemma \ref{lem:CARD-lifting}, we build a sequence $(M_k: k \in \mathbb{N})$ as follows.

Let $(\varphi_i(x,y,z): i \in \mathbb{N})$ enumerate all $\Sigma_{n+1}$-formulae with all free variables displayed.
Let $M_0 = M$.
Suppose that $k \in \mathbb{N}$ and $(M_\ell: \ell \leq k)$ are defined s.t.
each $M_\ell$ is a countable $\Sigma_{n+1}$-elementary cofinal extension of $M_{\ell-1}$ if $\ell > 0$
and $M_k \models I\Sigma_n + \neg \CARD(\Sigma_{n+1})$.
Fix an enumeration $(d_{j,\ell}: j \in \mathbb{N})$ of $M_\ell$.
Pick the least triple $(i,j,\ell)$ s.t. $\ell \leq k$ and 
\begin{equation}\label{eq:CARD-GPHP}
  M_k \models \exists w (\varphi_i(x,y,d_{j,\ell}) \text{ defines an injection } F_k: M_k \to w).
\end{equation}
Let $d \in M_k$ be an upper bound of the range of $F_k$.
The restriction of $F_k$ on $b_k^{b_k} d$ is a $\Sigma_{n+1}^{M_k}$-injection from $b_k^{b_k} d$ to $d$.
By an application of \cite[Lemma 5.3]{BCWWY:2021}, 
we get a countable $M_{k+1}$ s.t. 
\begin{itemize}
  \item $M_{k+1} = M_k[c_k]$ for some $c_k$, 
  \item $M_{k+1}$ is a $\Sigma_{n+1}$-elementary cofinal extension of $M_k$, 
  \item $[0,b_k]^{M_{k+1}} = [0,b_k]^{M_k}$, and 
  \item $M_{k+1} \models \forall y \neg \varphi_i(e,y,d_{j,\ell})$ for some $e < b_k^{b_k} d$.
\end{itemize}
By inductive applications of Lemma \ref{lem:CARD-lifting},
there is a $\Sigma^{M_{k+1}}_{n+1}$-injection from $M_{k+1}$ to $a$.
In particular $M_{k+1} \not\models \CARD(\Sigma_{n+1})$.
By Proposition \ref{pro:cf-ext-induction}, $M_{k+1} \models I\Sigma_n$.

Let $N = \bigcup_{k \in \mathbb{N}} M_k$.
Then $N$ is a $\Sigma_{n+1}$-elementary cofinal extension of $M$ and thus $N \models I\Sigma_n$.
To see that $N \models \CARD(\Sigma_{n+1})$, assume that $(i,j,\ell)$ is the least triple s.t.
\[
  N \models \varphi_i(x,y,d_{j,\ell}) \text{ defines an injection } F: N \to d
\]
for some $d \in N$.
By elementarity, we can pick $k \geq \ell$ s.t. $d \in M_k$ and $(i,j,\ell)$ is the least triple satisfying \eqref{eq:CARD-GPHP}.
So there exists $e \in M_{k+1}$ s.t.
$
  M_{k+1} \models \forall y \neg \varphi_i(e,y,d_{j,\ell}),
$
although the construction of $M_{k+1}$ may use an upper bound of the range of $F$ different from the above $d$.
By elementarity again,
$
  N \models \forall y \neg \varphi_i(e,y,d_{j,\ell})
$
as well, contradicting the above assumption.
So $N \models \CARD(\Sigma_{n+1})$.

As mentioned above,
there exists a $\Sigma_{n+1}^{M_k}$-injection from $M_k$ to $a$ for each $k$.
In particular, there exists a $\Sigma_{n+1}^{M_k}$-injection from $b_k$ into $a$.
Since $[0,b_k]^{M_k} = [0,b_k]^N$ and $N$ is a $\Sigma_{n+1}$-elementary extension of $M_k$,
 there also exists a $\Sigma_{n+1}^N$-injection from $b_k$ into $a$.
So $N \not\models \GPHP(\Sigma_{n+1})$.
\end{proof}

\section{Finite Ramsey Theorem and Weak Pigeonhole Principle}

This section presents another separation.

\begin{theorem}\label{thm:FRT-WPHP}
For every countable $M \models I\Sigma_n$ with $n > 0$ and any $a \in M$,
there exists $N$ s.t. $N$ is a $\Sigma_{n+1}$-elementary cofinal extension of $M$, 
$[0,a]^M = [0,a]^N$ and $N \models I\Sigma_n + \FRT(\Sigma_{n+1})$.

Thus,
$I\Sigma_n + \FRT(\Sigma_{n+1}) \not\vdash \WPHP(\Sigma_{n+1})$.
\end{theorem}

Recall that the following is a $I\Sigma_1$-provably total function
\[
  R(x,e,k) = \min\{y: y \to (x)^e_k\}.
\]
In $I\Sigma_1$, define
\[
  R^{(0)}(x,e,k) = x,\quad
  R^{(m+1)}(x,e,k) = R(R^{(m)}(x,e,k),e,k).
\]

The construction of $N$ in the above theorem needs the following technical lemma.

\begin{lemma}\label{lem:FRT-WPHP}
Suppose that 
\begin{itemize}
 \item[(a)] $M$ is a countable model of $I\Sigma_n$ with $n > 0$,
 \item[(b)] $a,b,c,e,x$ are elements of $M$, $b > \mathbb{N}$, $k = \max\{a,c\}$,
 \item[(c)] $\varphi$ is a $\Sigma_n^M$-formula defining $C: [y]^e \times M \to c$ 
 with $y = R^{(b)}(x,e,k)$, and
 \item[(d)] $\bar{C}(\vec{d}) = \lim_s C(\vec{d},s)$ exists for all $\vec{d} \in [y]^e$.
\end{itemize}
Then either one of the followings holds,
\begin{enumerate}
 \item[(1)] There exist $s \in M$, $H \in [y]^x \cap M$ and $i < c$, s.t.
 $\bar{C}(\vec{d}) = C(\vec{d},t) = i$ for all $\vec{d} \in [H]^e$ and $t > s$.
 \item[(2)] There exists $N$, s.t. $N$ is a $\Sigma_{n+1}$-elementary cofinal extension of $M$, 
 $[0,a]^N = [0,a]^M$, $\lim_s C^N(\cdot,s)$ is undefined on some $\vec{d} \in [y]^e \cap N$ where $C^N$ is the map defined by $\varphi$ in $N$.
\end{enumerate}
\end{lemma}

\begin{proof}
Fix parameters satisfying (a-d).
We assume that clause (1) fails and 
define an $M$-ultrafilter on $[y]^e$ s.t. $N = \Ult_U(M)$ satisfies clause (2).
Below we work in $M$.
Firstly, we prove a claim.

\medskip

\begin{claim}
If $s$ is any element and
$X$ is an $M$-finite subset of $y$ with cardinality at least some $R^{(m)}(x,e,k)$ with $m > 1$
then there exists an $M$-finite subset $Y$ of $X$ s.t.
$|Y| \geq R^{(m-1)}(x,e,k)$ and
\[
  \forall \vec{d} \in [Y]^e \exists t > s (C(\vec{d},t) \neq C(\vec{d},s)).
\]
\end{claim}

To prove the claim, let
\[
  A = \left\{ \vec{d} \in [X]^e: \forall t > s, i,j
    \left( C(\vec{d},s) = i \wedge C(\vec{d},t) = j \to i = j \right) \right\}.
\]
Then $A$ is $\Pi_n^M$ and thus $M$-finite.
Define $D: [X]^e \to 2$ as follows
\[
  D(\vec{d}) =
  \begin{cases}
    1 & \vec{d} \in A \\
    0 & \vec{d} \not\in A
  \end{cases}.
\]
So $D$ is $M$-finite as well.
By the definition of $R$, 
$X$ contains an $M$-finite subset $Y$ s.t. $|Y| \geq R^{(m-1)}(x,e,k)$ and $Y$ is $D$-homogeneous.
If $D([Y]^e) = \{1\}$ then
\[
  \left\{ (\vec{d},\bar{C}(\vec{d})): \vec{d} \in [Y]^e \right\} =
  \left\{ (\vec{d},C(\vec{d},s)): \vec{d} \in [Y]^e \right\},
\]
and thus the restriction of $\bar{C}$ to $[Y]^e$ is an $M$-finite coloring.
As $m-1 > 0$, this would imply $Y$ containing an $M$-finite $H$ satisfying clause (1).
So $D([Y]^e) = \{0\}$ and $Y$ satisfies the conclusion of the claim.

\medskip

We build a sequence $(X_i: i \in \mathbb{N})$ as follows.
Let $(f_i: i \in \mathbb{N})$ enumerate all $M$-finite $f: [y]^e \to a$,
and let $(s_i: i \in \mathbb{N})$ enumerate $M$.

Let $X_0 = y = [0,y-1]^M$.
Suppose that $i \in \mathbb{N}$ and $X_{2i}$ is an $M$-finite subset of $X_0$ and $|X_{2i}| \geq R^{(b-2i)}(x,e,k)$.
By the definition of $R$ and $k \geq a$, there exists an $M$-finite $X_{2i+1} \subseteq X_{2i}$ s.t.
 $f_i([X_{2i+1}]^e) = \{a_i\}$ for some $a_i < a$ and $|X_{2i+1}| \geq R^{(b-2i-1)}(x,e,k)$.
Apply the above claim to $(s,X,m) = (s_i,X_{2i+1},b-2i-1)$,
 we get an $M$-finite $X_{2i+2} \subseteq X_{2i+1}$ s.t.
 $|X_{2i+2}| \geq R^{(b-2i-2)}(x,e,k)$ and
\[
  \forall \vec{d} \in [X_{2i+2}]^e \exists t > s_i (C(\vec{d},t) \neq C(\vec{d},s_i)).
\]

Let $U$ be any $M$-ultrafilter on $[y]^e$ containing every $[X_i]^e$ and let $N = \Ult_U(M)$.
So $N \models I\Sigma_n$ and $N$ is a $\Sigma_{n+1}$-elementary cofinal extension of $M$.
By Lemma \ref{lem:Ult-reprstn}, $N = M[\vec{d}]$ for some $\vec{d}$ s.t. $N \models \vec{d} \in [X_{i}]^e$ for all $i$.

If $N \models a' < a$ then $a' = f_i(\vec{d})$ for some $f_i$ in the above enumeration.
As $M \models f_i([X_{2i+1}]^e) = \{a_i\}$ and $N \models \vec{d} \in [X_{2i+1}]^e$, $a' = a_i \in M$.
Thus $[0,a]^N = [0,a]^M$.

As $N \models \vec{d} \in [X_{2i+2}]^e$, $N \models \exists t > s_i (C(\vec{d},t) \neq C(\vec{d},s_i))$ by the construction of $X_{2i+2}$. Thus $\lim_s C^N(\vec{d},s)$ does not exist.
\end{proof}

\begin{proof}[Proof of Theorem \ref{thm:FRT-WPHP}]
Fix a countable $M \models I\Sigma_n$ with $n > 0$ and $a \in M$ and $b \in M - \mathbb{N}$.
Also fix an enumeration $(\varphi_i(u,s,v,w): i \in \mathbb{N})$ of $\Sigma_n$-formulae with all free variables displayed.

We build a sequence of countable models $(M_q: q \in \mathbb{N})$ as follows.

Let $M_0 = M$.

Suppose that $M_p$'s are built and $M_p \models I\Sigma_n$ for $p \leq q$.
Let $((x_{j,p},e_{j,p},c_{j,p},d_{j,p}): j \in \mathbb{N})$ be an enumeration of quadruples in $M_p$ for $p \leq q$.
For each $(j,p)$, let $k_{j,p} = \max\{a,c_{j,p}\}$ and $y_{j,p} = R^{(b)}(x_{j,p},e_{j,p},k_{j,p})$ be as calculated in $M_q$.
Let $(i,j,p)$ be the least triple s.t. $p \leq q$ and in $M_q$ the following holds
\begin{multline}\label{eq:FRT-WPHP}
  \varphi_i(u,s,v,d_{j,p}) \text{ defines a function } C_{i,j,p}: [y_{j,p}]^{e_{j,p}} \times M_q \to c_{j,p}
  \text{ and }\\
   \forall \vec{d} \in [y_{j,p}]^{e_{j,p}} (\bar{C}_{i,j,p}(\vec{d}) = \lim_s C_{i,j,p}(\vec{d},s) \text{ exists}) \text{ and }\\
   \forall H \in [y_{j,p}]^{x_{j,p}} (H \text{ is not homogeneous for } \bar{C}_{i,j,p}).
\end{multline}
Apply Lemma \ref{lem:FRT-WPHP} to $(M_q,a,b,c_{j,p},e_{j,p},x_{j,p},C_{i,j,p})$,
and note that the last line of \eqref{eq:FRT-WPHP} implies the failure of Lemma \ref{lem:FRT-WPHP}(1),
so let $M_{q+1}$ be the resulted model in Lemma \ref{lem:FRT-WPHP}(2).

Let $N = \bigcup_{q \in \mathbb{N}} M_q$.
Then $N$ is a $\Sigma_{n+1}$-elementary cofinal extension of every $M_q$, $[0,a]^N = [0,a]^M$ and $N \models I\Sigma_n$.

Assume that $N \not\models \FRT(\Sigma_{n+1})$.
Let $(i,j,p)$ be the least triple s.t. \eqref{eq:FRT-WPHP} holds with $M_q$ replaced by $N$.
By elementarity, pick $q \geq p$ s.t. $(i,j,p)$ is also the least triple satisfying \eqref{eq:FRT-WPHP} in $M_q$.
By the construction of $M_{q+1}$, $M_{q+1}$ satisfies
\begin{equation}\label{eq:FRT-WPHP-1}
  \exists \vec{d} \in [y_{j,p}]^{e_{j,p}} (\bar{C}_{i,j,p}(\vec{d}) \text{ is undefined}).
\end{equation}
As $N$ is a $\Sigma_{n+1}$-elementary extension of $M_{q+1}$, $N$ also satisfies \eqref{eq:FRT-WPHP-1}.
Thus the second line of \eqref{eq:FRT-WPHP} cannot hold in $N$.
This contradiction shows that $N \models \FRT(\Sigma_{n+1})$.

Finally, if in the model $M$ above there exists a $\Sigma_{n+1}^M$-injection $F: 2a \to a$,
then $[0,2a]^N = [0,2a]^M$ as $[0,a]^N = [0,a]^M$,
and the $\Sigma_{n+1}^M$-formula defining $F$ in $M$ also defines the same injection in $N$ by elementarity.
Hence
\[
  N \models I\Sigma_n + \FRT(\Sigma_{n+1}) + \neg \WPHP(\Sigma_{n+1}).\qedhere
\]
\end{proof}

The next corollary follows easily.

\begin{corollary}
Over $I\Sigma_{n}$ ($0 < n \in \mathbb{N})$,
$\GPHP(\Sigma_{n+1})$ is strictly between $\CARD(\Sigma_{n+1})$ and $\WPHP(\Sigma_{n+1})$.
\end{corollary}

\begin{proof}
As parts of the known results summarized in Theorem \ref{thm:knowns},
\[
  \WPHP(\Sigma_{n+1}) \to \GPHP(\Sigma_{n+1}) \to \CARD(\Sigma_{n+1}).
\]
By Theorem \ref{thm:GPHP-FRT}, $I\Sigma_n \vdash \GPHP(\Sigma_{n+1}) \leftrightarrow \FRT^1(\Sigma_{n+1})$.
By Theorem \ref{thm:FRT-WPHP}, $I\Sigma_n + \FRT^1(\Sigma_{n+1}) \not\vdash \WPHP(\Sigma_{n+1})$,
thus $I\Sigma_n + \GPHP(\Sigma_{n+1}) \not\vdash \WPHP(\Sigma_{n+1})$.
By Theorem \ref{thm:CARD-GPHP}, $I\Sigma_n + \CARD(\Sigma_{n+1}) \not\vdash \GPHP(\Sigma_{n+1})$.
\end{proof}

\section{Finite Combinatorics and Approximations}

Recall that for a partial $\Sigma_n$-function $f$ and a number $b$,
an \emph{approximation} to $f$ of length $b$ is a sequence $s$,
 s.t., $s$ is of length $b$ and
\[
  \forall i < b-1 \forall x \leq s(i) \forall y(y = f(x) \to y \leq s(i+1)).
\]
For $n > 0$, $P\Sigma_{n}$ is the $\Sigma_{n+2}$-sentence stating that for every partial $\Sigma_n$-function $f$ and every $b$ there exists an approximation to $f$ of length $b$.
Like $B\Sigma_{n+1}$, $P\Sigma_n$ also lies strictly between $I\Sigma_n$ and $I\Sigma_{n+1}$.
Moreover, $P\Sigma_n$ and $B\Sigma_{n+1}$ are independent over $I\Sigma_n$, and even $B\Sigma_{n+1} + P\Sigma_{n+1}$ is strictly weaker than $I\Sigma_{n+1}$.
For details, the reader can consult \cite[Chapter IV]{HP:1998.book}.

In particular, $P\Sigma_1$ has proven important in the reverse mathematics of Ramsey's Theorem for pairs and related combinatorial principles (e.g., see \cite{KrYo:2016.JML, CWY:2024.TT, HoPaYo:2025.TAMS}).
This might make it natural to study the finite combinatorial principles formalized here in the presence of $P\Sigma_n$.

\begin{theorem}\label{thm:PSigma}
For positive integers $n$,
\[
  I\Sigma_n + P\Sigma_n + \WPHP(\Sigma_{n+1}) + \FRT(\Sigma_{n+1}) \not\vdash B\Sigma_{n+1}.
\]
\end{theorem}

The proof of this theorem is based on the following two lemmata,
which are from an unfinished joint project of Chong, Yang and the author.

Firstly, a variant of approximations to partial functions is needed.
For a partial $\Sigma_n$-function $f$ and a total function $g$,
 a \emph{$g$-stretched approximation} to $f$ of length $b$ is a sequence $s$,
 s.t., $s$ is of length $b$ and
\[
  \forall i < b-1 \forall x \leq g(s(i)) \forall y(y = f(x) \to y \leq s(i+1)).
\]
Stretched approximations lead to a slightly different formulation of $P\Sigma_n$.

\begin{lemma}[Chong, Wang and Yang]\label{lem:PSigma-variants}
For $n > 0$, the followings are equivalent over $I\Sigma_n$,
\begin{enumerate}
 \item $P\Sigma_n$;
 \item For every partial $\Sigma_n$-function $f$ and every pair $a$ and $b$,
  there exists an approximation $s$ to $f$ s.t. $s(0) \geq a$ and $s$ is of length $b$;
 \item If $f$ is a partial $\Sigma_n$-function and $g$ is a total $\Sigma_n$-function,
  then for every $b$ there exists a $g$-stretched approximation to $f$ of length $b$.
\end{enumerate}
\end{lemma}

\begin{proof}
It suffices to prove that (1) implies each of the others.
So we take an arbitrary $M \models P\Sigma_n$.

For a given partial $\Sigma_n$-function $f$ and a given pair $a$ and $b$,
 define
\[
  g(0) = a, \quad g(x+1) = f(x).
\]
Let $t$ be an approximation to $g$ of length $b+1$.
Then the sequence defined below is an approximation to $f$ of length $b$,
\[
  s(i) = t(i+1).
\]
This proves that (2) holds in $M$.

To prove (3) in $M$, fix $f$, $g$ and $b$ as in (3).
We may assume that $g$ is strictly increasing.
Let $h = g \circ f$. So $h$ is a partial $\Sigma_n$-function.
By (2), let $t$ be an approximation to $h$ of length $b$ with $t(0) \geq g(0)$.
Let $F$ and $G$ be the restrictions of $f$ and $g$ to $[0,t(b-1)]$ respectively.
By $I\Sigma_n$, both $F$ and $G$ are $M$-finite.
We define a sequence $s$ by induction.
\begin{enumerate}
 \item $s(0) = 0$.
 \item If $0 < i < b$ and $s(i-1)$ is defined, then let
\[
  s(i) = \max \left( \{F(x): x \leq G(s(i-1)) \wedge x \in \dom F\} \cup \{s(i-1)\} \right).
\]
\end{enumerate}
We prove that $g(s(i)) \leq t(i)$ for each $i < b$.
By definition, $g(s(0)) \leq t(0)$.
Suppose that $0 < i < b$ and $g(s(i-1)) \leq t(i-1)$.
If $x \leq g(s(i-1))$ and $f(x)$ is defined, then $h(x) \leq t(i)$ by the choice of $t$.
Hence $g(s(i)) \leq t(i)$.
So the inductive construction of $s$ can be done with the help of $F$ and $G$ in $M$.
It follows from the construction that $s$ is a $g$-stretched approximation to $f$ of length $b$.
\end{proof}

By the next lemma, the satisfaction of $P\Sigma_n$ cannot be changed by cofinal extensions. 
\footnote{A similar result is independently obtained by Mengzhou Sun and Tin Lok Wong.}

\begin{lemma}[Chong, Wang and Yang]\label{lem:Approx-Cof-ext}
Suppose that $M \models I\Sigma_n$ and $N$ is a $\Sigma_{n+1}$-elementary cofinal extension of $M$.
Then
\[
  M \models P\Sigma_n \text{ if and only if } N \models P\Sigma_n.
\]
\end{lemma}

\begin{proof}
Suppose that $N \models P\Sigma_n$.
Since $P\Sigma_n$ is a $\Pi_{n+2}$-statement and $M$ is a $\Sigma_{n+1}$-elementary submodel of $N$,
 $M \models P\Sigma_n$ as well.

Suppose that $M \models P\Sigma_n$.
Let $\varphi(w,x,y)$ be a $\Sigma_n$-formula with all free variables displayed and $a \in N$ be s.t. $\varphi(a,x,y)$ defines a partial function in $N$.
Fix $\bar{a} \in M$ with $\bar{a} > a$.
Let
\[
  A = \{k \in M: k < \bar{a}, M \models \varphi(k,x,y) \text{ defines a partial function}\}.
\]
Then $A$ is $\Pi_n^M$ and bounded, thus $A$ is $M$-finite.
By elementarity, $N \models a \in A$.
For each $k \in A$, let $F_{\varphi,k}$ denote the partial function defined by $\varphi(k,x,y)$.
Let $\psi(k,x,y)$ be $k \in A \wedge \varphi(k,x,y)$.
Then $\psi \in \Sigma_n^M$ and defines a partial function $F_\psi$ in $M$ sending $(k,x)$ to $y$.
Let $g(x) = (\bar{a},x)$.
Fix $b \in M$.
By Lemma \ref{lem:PSigma-variants}, in $M$ there exists a $g$-stretched approximation $s$ to $F_\psi$ of length $b$.
For $i < b-1$, $x \leq s(i)$ and $k \in A$, $(k,x) \leq g(s(i))$;
by the choice of $s$, if $y = F_{\varphi,k}(x) = F_\psi(k,x)$, then $y \leq s(i+1)$.
Thus
\[
  M \models \forall i < b-1 \forall k \in A \forall x \leq s(i) \forall y
  (\varphi(k,x,y) \to y \leq s(i+1)).
\]
By elementarity, the same holds in $N$ as well.
Thus, in $N$, $s$ is an approximation to the partial function defined by $\varphi(a,x,y)$ of length $b$.
As $M$ is cofinal in $N$, this shows that $N \models P\Sigma_n$.
\end{proof}

Fix a countable $M \models P\Sigma_n + \neg B\Sigma_{n+1}$.
By combining the proofs of \cite[Theorem 5.10]{BCWWY:2021} and Theorem \ref{thm:FRT-WPHP} here,
there exists $N$, s.t., $N$ is a $\Sigma_{n+1}$-elementary cofinal extension of $M$ and 
\[
  N \models I\Sigma_n + \WPHP(\Sigma_{n+1}) + \FRT(\Sigma_{n+1}) + \neg B\Sigma_{n+1}.
\]
By the lemma above, $N \models P\Sigma_n$ as well.

So Theorem \ref{thm:PSigma} is proved.

\section{Questions}

Personally, this research started because of \cite{BCWWY:2021},
where it is shown that the first order theory of $\RCA_0 + 2\text{-}\RAN$ is below $I\Sigma_1 + \WPHP(\Sigma_2)$,
where $2\text{-}\RAN$ is the $\Pi^1_2$-sentence in second order arithmetic saying that for every $X$ there exists $2$-random $Y$ relative to $X$.
The separation of $\CARD(\Sigma_2)$, $\GPHP(\Sigma_2)$ and $\WPHP(\Sigma_2)$ here naturally leads to the question whether $\RCA_0 + 2\text{-}\RAN$ implies $\GPHP(\Sigma_2)$.

Another kind of questions is about the relation between $\FRT^e_k(\Sigma_n)$.
According to the above sections, based on $I\Sigma_n$, the following implications are known,
\[
  B\Sigma_{n+1} \to \WPHP(\Sigma_{n+1}) \to \FRT^e_k(\Sigma_{n+1}) \to 
   \FRT^1(\Sigma_{n+1}) \leftrightarrow \GPHP(\Sigma_{n+1}),
\]
and so are the following non-implications
\[
  B\Sigma_{n+1} \not\leftarrow \WPHP(\Sigma_{n+1}) \not\leftarrow \FRT^e_k(\Sigma_{n+1}) \not\leftarrow \FRT^1(\Sigma_{n+1}),
\]
where both $e$ and $k$ are integers greater than $1$.
But the relations among $\FRT^e_k(\Sigma_{n+1})$'s remain unknown.

\bibliographystyle{plain}
\bibliography{def-comb}

\end{document}